\documentclass[11pt, reqno]{amsart}
\usepackage{amssymb,latexsym}
\usepackage{color,dsfont}
\usepackage{xcolor}
\usepackage{amsmath}
\usepackage{mathrsfs}
\usepackage{amsthm}
\usepackage{graphicx}
\usepackage{verbatim}
\usepackage{enumerate}
\usepackage[mathscr]{eucal}
\usepackage{lineno}

\theoremstyle{plain}
\newtheorem{theorem}{Theorem}
\newtheorem{proposition}{Proposition}
\newtheorem{corollary}{Corollary}
\newtheorem{lemma}{Lemma}

\theoremstyle{definition}

\renewcommand{\Re}{\operatorname{Re}}

\allowdisplaybreaks

\begin{document}
\title{Small gaps and small spacings between zeta zeros}
\author{Hung M. Bui}
\address{Department of Mathematics, University of Manchester, Manchester M13 9PL, UK}
\email{hung.bui@manchester.ac.uk}
\author{Daniel A. Goldston}
\address{Department of Mathematics and Statistics \\ San Jose State University \\ San Jose, CA 95192, USA}
\email{daniel.goldston@sjsu.edu}
\author{Micah B. Milinovich}
\address{Department of Mathematics, University of Mississippi, University MS 38677, USA}
\email{mbmilino@olemiss.edu}
\author{Hugh L. Montgomery}
\address{Department of Mathematics, University of Michigan, Ann Arbor, MI  48109, USA}
\email{hlm@umich.edu}
\subjclass[2010]{11M06, 11M26.}
\keywords{Riemann zeta-function, distinct zeros, zero spacing, small gaps, pair correlation, moments}

\date{\today}

\dedicatory{Dedicated to Andrzej Schinzel, with appreciation for his contributions and gratitude for his many years of service to number theory.}

\maketitle
\section{Introduction}
    We assume the Riemann hypothesis (RH) throughout this paper. Let $ 1/2 +i\gamma$ denote a nontrivial zero of the Riemann zeta-function, $\zeta(s)$, and let $m(\gamma)$ denote its multiplicity. While we expect that these zeros are all simple with $m(\gamma)=1$, currently we cannot exclude the existence of multiple zeros. In this paper, we address how the possible existence of multiple zeros affects the results we can prove on close pairs of zeros. We consider both the non-decreasing sequence $\{\gamma\}$
of positive ordinates $\gamma >0$ which counts multiplicity, and also the increasing sequence $\{\gamma_d\}$  of distinct zeros with ordinates $\gamma_d >0$ which does not count multiplicity. If all the zeros are simple then these sequences are identical.   As usual, we let $N(T)$ denote the number of zeros with $0<\gamma \le T$, counting multiplicity. Then
\begin{equation} \label{N(T)}
N(T) := \sum_{0< \gamma \le T} 1 = \sum_{0< \gamma_d \le T} m(\gamma_d) = \frac{T}{2\pi} \log \frac{T}{2\pi} - \frac{T}{2\pi} + \frac{7}{8} + S(T) + O\Big(\frac{1}{T}\Big)
\end{equation}
for $ T \ge 2$, where the remainder term $S(T)=\frac{1}{\pi}\arg \zeta\!\left(\tfrac{1}{2}+iT\right)$, for $T\neq \gamma$, with the argument obtained by a continuous variation along the line segments joining the points $2,$ $2+iT$ and $\frac{1}{2}+iT$ starting with the value $\arg \zeta(2)=0.$ It is known that $S(T) = O(\log T)$, as $T\to \infty$.

\smallskip

     There are two possible types of close pairs of zeros:~close pairs from two distinct zeros which have a positive distance between them and close pairs that occur from the same multiple zero having distance zero from each other. Of course, two distinct zeros can both be multiple zeros generating many pairs of zeros of each type. We use the word {\it spacing} between zeros for both types of close pairs of zeros, and reserve the word {\it gap} between pairs of zeros to mean a spacing with a strictly positive length.  Denote by $\gamma^+$ the next term $\gamma \le \gamma^+$ after $\gamma$ in the sequence of ordinates of zeros. Similarly, we denote by  $\gamma_d^+$ the next term $\gamma_d <\gamma_d^+$ in the sequence of distinct ordinates. By \eqref{N(T)}, the average of the consecutive spacings $\gamma^+-\gamma$ is $2\pi/\log \gamma$, and to measure how close these zeros can get we define
 \begin{equation}\label{mu}
 \mu := \liminf_{\gamma \to \infty} \, (\gamma^+ -\gamma) \, \frac{\log \gamma }{2\pi} .
 \end{equation}
Similarly, to measure small gaps between consecutive distinct zeros we define
 \begin{equation}\label{mu*}
 \mu_d := \liminf_{{\gamma_d}\to \infty}\, (\gamma_d^+ -\gamma_d) \, \frac{\log \gamma_d }{2\pi} .
 \end{equation}

 \smallskip

More generally, we consider the distribution functions
\begin{equation} \label{D(lambda)} D(\lambda, T) :=  \frac{1}{N(T)}  \sum_{\substack{0<\gamma \le T\\ \gamma^+ -\gamma \le \frac{2 \pi \lambda}{\log T}}} \!\!\!\!\!  1 \quad \text { and } \quad D_d(\lambda, T) :=  \frac{1}{N(T)}  \sum_{\substack{0<\gamma_d \le T\\ \gamma_d^+ -\gamma_d \le \frac{2 \pi \lambda}{\log T}}} \!\!\!\!\! 1. \end{equation}
Corresponding to \eqref{mu} and \eqref{mu*}, we define
\begin{equation*}  \mu_D :=\inf_\lambda\left\{  \lambda : \liminf_{T\to \infty} D(\lambda, T) > 0\right\} \quad \text{and} \quad \mu_{D_d} :=\inf_\lambda\left\{  \lambda : \liminf_{T\to \infty} D_d(\lambda, T) > 0\right\} . \end{equation*}
Thus, if $\lambda > \mu_D$, then there are a positive proportion of spacings of consecutive zeros of length $\le \lambda$ times the average spacing. Likewise, if $\lambda > \mu_{D_d}$, then there are a positive proportion of gaps between distinct consecutive zeros with length $\le \lambda$ times the average spacing. Note trivially that
\begin{equation*}
\mu \le \mu_d \le \mu_{D_d}\quad \text{and} \quad  \mu \le \mu_D \le \mu_{D_d}.
\end{equation*}

\smallskip

 There are three known methods for proving the existence of close pairs of zeros.  The first method is due to Selberg who proved unconditionally that, for a small positive constant $\delta$,
 \begin{equation}\label{Selberg} \mu_D \le 1-\delta. \end{equation}
 Selberg's proof is based on moments of the remainder term $S(T)$ in
 \eqref{N(T)}. Though his proof was never published, Fujii \cite{Fujii75, Fujii81} gave an abbreviated proof of \eqref{Selberg} and a more detailed proof was given by Heath-Brown in the appendix to  \cite[Chapter 9]{T}. For references to these results and corrections to some misprints that have occurred, see \cite{CT-B}. Selberg's proof is in two steps. First, it is proved that a positive proportion of consecutive zeros have gaps larger than the average spacing, and then this result is used to infer  a positive proportion of consecutive zeros with spacing less  than the average spacing. In \cite{STT-B}, the value $\delta= \frac{1}{2}\times 10^{-3\cdot 10^{13}}$ is obtained.

\smallskip

Another method, which depends on RH, was introduced by Montgomery and Odlyzko \cite{MO} who used it to obtain the estimate $\mu \le 0.5179$. There have been many refinements made (e.g.~\cite{BMN,CGG84,FW}) and the current best result is due to Preobrazhenski\u\i \  \cite{P}, who proved that
\[
\mu \le 0.515396.
\]
Two different proofs of this method are given in \cite{MO} or \cite{CGG84} and the limitations of the method are discussed in \cite{GTTB}. This method has also been modified by Conrey, Ghosh, Goldston, Gonek, Heath-Brown \cite{CGGGHB} in 1985 to produce the positive proportion result $\mu_D \le 0.77$. Soundararajan \cite{So} refined this method to obtain $\mu_D \le 0.6876$
, and Wu \cite{Wu} later obtained \begin{equation*} \mu_D\le 0.6653. \end{equation*}

\smallskip

A third method for finding close pairs of zeros is due to Montgomery \cite{M72} in his paper on pair correlation of zeros, and also assumes RH. Montgomery originally obtained the estimate $\mu \le 0.68$. This method has been refined by other authors (e.g.~\cite{CCLM,GGOS}) with the current best result
\begin{equation} \label{PCmu}
\mu \le 0.6039,
\end{equation}
due to Chirre, Gon{\c{c}}alves, and de Laat \cite{CGdL}.  In this paper, we prove that Montgomery's pair correlation method actually gives a result for a positive proportion of spacings, improving the best known bounds from the previous method used in \cite{CGGGHB,So,Wu}.
\begin{theorem}\label{th1}
Assuming RH, we have
\[
\mu_D \le  0.6039.
\]
\end{theorem}

\smallskip

The results mentioned above, for all three methods, fail to exclude the possibility that the small spacings that are detected are composed entirely from zero spacings between multiple zeros (and not from actual gaps). We suspect that all three methods cannot be modified to prove the existence of gaps between zeros smaller than the average spacing, in other words that all three of these methods are incapable of proving that $\mu_d <1$. In this paper, we introduce a new method specifically designed to find small gaps between distinct zeros.

\begin{theorem} \label{thm4} Assuming RH, we have
\[
\mu_d \le 0.991.
\]
\end{theorem}

The method we introduce produces gaps between a simple zero and a distinct zero of odd multiplicity. Both configurations of $\gamma_d$ and $\gamma_d^+$, where $\gamma_d$ is a simple zero or $\gamma_d^+$ is a simple zero, occur. However, we cannot guarantee that both of these distinct zeros are simple. Furthermore, our method does not produce a positive proportion of gaps and so we do not obtain a result for $\mu_{D_d}$. Nevertheless, in an interval $[T,2T]$, we can show that there are $\gg_\varepsilon T^{1-\varepsilon}$ such small gaps between distinct zeros, for any $\varepsilon>0$.

 \begin{corollary}\label{corollaryquantitativeresult}
Assume RH and let $T$ be large. Then, for any constant $C> \log 4$, there are
\[
\gg \ T \, \exp\!\Big( - C \frac{\log T}{\log\log T} \Big)
\]
consecutive ordinates $\gamma_d,\gamma_d^+ \in [T,2T]$ of distinct zeros with $\gamma_d < \gamma_d^+$ and $\gamma_d^+-\gamma_d \le 0.991 \frac{2\pi}{\log T}$.
\end{corollary}

We conclude the introduction by mentioning that this paper was inspired, in part, by the recent work of Rodgers and Tao \cite{RT} who proved that de Bruijn--Newman constant is non-negative. The last step of their proof relied on knowing that $\mu_D <1$. There are also other reasons for studying small gaps and small spacings between zeta zeros. For instance, there is a well-known connection between the existence of small spacings between the zeros of the zeta function and the class number problem for imaginary quadratic fields. See the works of Conrey and Iwaniec in \cite{CI} and Montgomery and Weinberger \cite{MW} for more on this connection. For an overview of results on the complementary problem of proving the existence of large gaps between zeros of the zeta function, see \cite{BM} and the references therein.

\section{Pair correlation of zeta zeros and the proof of Theorem \ref{th1}}
In this section, we investigate small spacings and small gaps between zeta zeros using Montgomery's pair correlation method \cite{M72}. As usual, we define the form factor
\begin{equation*}
F(\alpha) = F(\alpha,T) = \frac{1}{N(T)}\sum_{0<\gamma, \gamma' \le T}
T^{i\alpha(\gamma - \gamma')}w(\gamma-\gamma'),
\end{equation*}
where $\alpha$ and $T\ge 2$ are real, $w(u) =  4/(4+u^2)$.  By Fourier inversion, for any function $r \in L^1(\mathbb{R})$ such that $ \widehat{r} \in L^1(\mathbb{R})$, we have
\begin{equation}\label{conv}
\sum_{0<\gamma,\gamma'\le T }r\Big((\gamma-\gamma')\frac{\log T}{2\pi}\Big) \, w(\gamma-\gamma')
= N(T)\int_{\mathbb R}\widehat r(\alpha)\,F(\alpha)\,d\alpha,
\end{equation}
where the Fourier transform $\widehat r$ is defined by
\[
\widehat r(\alpha) = \int_{\mathbb R} r(u) \, e(-\alpha u)\,du, \quad e(x):=e^{2\pi i x}.
\]
Assuming RH, it is known that $F$ is real-valued, even, nonnegative, and that
\begin{equation}\label{Montgomerythm}
F(\alpha) = \big(1+o(1)\big) \, T^{-2\alpha}\log T + \alpha + o(1)
\end{equation}
uniformly for $0\le |\alpha| \le 1$ (see \cite{GoldMont, M72}). Therefore, we can asymptotically evaluate the right-hand side of \eqref{conv} when $\mathrm{supp}\big(\widehat{r}\big) \subseteq [-1,1]$. By exploiting the fact that $F$ is nonnegative, we can further specialize our conditions on $r$ to prove the existence of close pairs of zeros.

\smallskip

For $\lambda > 0$, let $\mathcal{A}(\lambda)$ denote the class of even, continuous, and real-valued functions $r \in L^1(\mathbb{R})$ satisfying the following three conditions:
\begin{enumerate}
\item[($i$)] $r(0) = 1$;
\item[($ii$)]  $r(u) \le 0$ if  $|u|> \lambda$;
\item[($iii$)] $\widehat{r}(\alpha)\ge 0$ for all $\alpha \in \mathbb{R}$.
\end{enumerate}
It can be shown that if $r \in \mathcal{A}(\lambda)$, then $\widehat{r} \in L^1(\mathbb{R})$. We also let
\[
n^*  := \limsup_{T\to \infty} \frac{1}{N(T)}\sum_{0<\gamma_d \le T } m(\gamma_d)^2
\]
and, by \eqref{N(T)}, we note that $n^*\ge 1$. Then, the following theorem holds.

\begin{theorem}  \label{thm3} Assume RH. Let $r \in \mathcal{A}(\lambda)$ and define
\begin{equation} \label{c_lambda}
c(\lambda;r) := \widehat r(0) -1+ 2\int_0^1\alpha \,  \widehat r(\alpha)\,d\alpha.
\end{equation}
If there exists a $\lambda_0>0$ and an $r \in \mathcal{A}(\lambda_0)$ such that $c(\lambda_0;r)>0$ for sufficiently large $T$, then we have
\begin{equation} \label{thm3:1} D(\lambda_0,T) \gg 1 \qquad \text{and} \qquad \mu_D\le \lambda_0. \end{equation}
If there exists a $\lambda^*>0$ and an $r \in \mathcal{A}(\lambda^*)$ such that $c(\lambda^*;r)>n^*-1$ for sufficiently large $T$, then we have
\begin{equation}\label{mustar} D_d(\lambda^*,T) \gg c(\lambda^*;r)^4 \qquad \text{and}\qquad  \mu_{D_d}\le \lambda^*. \end{equation}
\end{theorem}

\smallskip

With appropriate choices for $r$, we obtain the following numerical values.

\smallskip

\begin{corollary} \label{cor1} Assuming RH, we have
\begin{equation*} \mu_D \le  0.6039, \qquad \text{and} \qquad \mu_{D_d} \le 1.0522. \end{equation*}
\end{corollary}

\smallskip

The result for $\mu_D$ follows from Theorem \ref{thm3} and the example of Chirre, Gon{\c{c}}alves, and de Laat \cite{CGdL} used to prove \eqref{PCmu}. If $r$ is contained in the class of functions $\mathcal{A}_{LP}$ described in \cite{CGdL}, then $r \in \mathcal{A}(\lambda)$ for any $\lambda\ge\inf\{ a>0 : r(x) \le 0 \text{ for } |x|\ge a\}$. The result for $\mu_{D_d} $ is  derived in \textsection \ref{sec4}. It may seem paradoxical that the bound we obtain for $\mu_{D_d} $ is larger than the average spacing between zeros, but presently (on RH) we only know that at least 84.77\% of the nontrivial zeros of $\zeta(s)$ are distinct, see \cite[Theorem 4]{CGdL}. In other words, the average spacing between distinct zeros could be as large as $1.17966 $ times the average spacing between all zeros, counting multiplicity. Therefore our result does find small gaps between distinct zeros with respect to the worst case scenario for the average spacing of distinct zeros.

\smallskip

Our proof of Theorem \ref{thm3} is based upon the following two propositions.

\begin{proposition} \label{thm1} Assume RH. If $r \in \mathcal{A}(\lambda)$,
then
\begin{equation}\label{thm1result}
\sum_{0<\gamma_d \le T } m(\gamma_d) \, \big(m(\gamma_d)-1\big) \ \ + \!\!\!\! \sum_{\substack{0<\gamma,\gamma'\le T \\ 0<|\gamma-\gamma'|\le \frac{2\pi\lambda}{\log T}}} \!\!\!\!\!\!\!\! 1  \ \ge \ \big(c(\lambda;r) -o(1) \big) \, N(T),
\end{equation}
where $c(\lambda;r)$ is defined in \eqref{c_lambda}.
\end{proposition}
If we can find an $r$ and a $\lambda_0>0$ so that $c(\lambda_0;r)>0$, then from \eqref{thm1result} we immediately conclude that $\mu \le \lambda_0$. In fact, the previous results on small gaps established using Montgomery's pair correlation method in \cite{CCLM,CGdL,GGOS,M72} all use Proposition \ref{thm1} to prove the existence of close pairs of zeros assuming RH. To deduce that one can obtain positive proportions for gap results from Proposition \ref{thm1}, we use the following result whose proof relies on techniques introduced by Selberg in the 1940s.

\begin{proposition}\label{thm2} Assume RH.
Suppose that $T\ge 2$ and that $k$ is a positive integer. Let
\begin{equation}\label{n(t,lambda)}
n(t,\lambda) = N\Big(t+\frac{2\pi \lambda}{\log T}\Big) - N(t).
\end{equation}
Then there is a positive absolute constant $C$ such that
\begin{equation} \label{zeromoment}
\sum_{0<\gamma\le T} n(\gamma,k)^{2k} \le (Ck)^{2k} \, T\log T
\end{equation}
and
\begin{equation}\label{multiplicitymoment}
\sum_{0<\gamma\le T} m(\gamma)^{2k-1} = \sum_{0<\gamma_d\le T} m(\gamma_d)^{2k} < (Ck)^{2k-1} \, T\log T\,.
\end{equation}
\end{proposition}

As described in \textsection \ref{proofthm2}, our proof of Proposition \ref{thm2} overlaps earlier results of a number of authors. We conclude this section with the proof of Proposition \ref{thm1}.

\begin{proof}[Proof of Proposition \ref{thm1}] For functions $r \in \mathcal{A}(\lambda)$, it follows that
\[
r(u) \le |r(u)| =  \left| \int_{\mathbb R} \widehat r(\alpha) \, e(\alpha u)\,d\alpha \right| \le   \int_{\mathbb R} |\widehat r(\alpha)|\,d\alpha =  \int_{\mathbb R} \widehat r(\alpha)\,d\alpha = r(0) = 1\,.
\]
Hence
\[
\sum_{\substack{0<\gamma,\gamma'\le T \\|\gamma-\gamma'|\le \frac{2\pi\lambda}{\log T}}} \!\!\!\!\!\!1 \ \
\ge \, \sum_{0<\gamma,\gamma'\le T } r\Big((\gamma-\gamma') \frac{\log T}{2\pi}\Big) \, w(\gamma-\gamma')= N(T)\int_{\mathbb R}\widehat r(\alpha) \, F(\alpha)\,d\alpha .
\]
Since $F$ and $\widehat r$ are nonnegative, the inequality is still valid if we restrict the integral on the right-hand side to any finite interval. Since $F$ and $\widehat r$ are even, by \eqref{Montgomerythm}, we see that
\[
\begin{split}
\sum_{\substack{0<\gamma,\gamma'\le T \\|\gamma-\gamma'|\le \frac{2\pi\lambda}{\log T}}} \!\!\!\!\!\!1 \ \
&\ge \ N(T)\int_{-1}^1\widehat r(\alpha)F(\alpha)\,d\alpha
\\
&= N(T) \bigg(\widehat r(0) + 2\int_0^1 \alpha \, \widehat r(\alpha)\,d\alpha +o(1) \bigg)
\\
&= N(T) \bigg( 1 + c(\lambda;r) +o(1) \bigg).
\end{split}
 \]
On the other hand, note that
\[
\sum_{\substack{0<\gamma,\gamma'\le T \\|\gamma-\gamma'|\le \frac{2\pi\lambda}{\log T}}} \!\!\!\!\!\!1 \ \
= \sum_{0<\gamma_d \le T } m(\gamma_d)^2 + \sum_{\substack{0<\gamma,\gamma'\le T \\ 0<|\gamma-\gamma'|\le \frac{2\pi\lambda}{\log T}}}\!\!\!\!\!\!1.
\]
Therefore
\[
\sum_{0<\gamma_d \le T } m(\gamma_d)^2 + \sum_{\substack{0<\gamma,\gamma'\le T \\ 0<|\gamma-\gamma'|\le \frac{2\pi\lambda}{\log T}}}\!\!\!\!\!\!1  \ \ \ge \, N(T) \, \bigg( 1 + c(\lambda;r) +o(1) \bigg).
\]
Making use of the identity
\[
N(T) = \sum_{0< \gamma_d \le T} m(\gamma_d)
\]
and then rearranging terms, the desired result in \eqref{thm1result} follows.
\end{proof}

\section{Proof of Proposition \ref{thm2}}\label{proofthm2}

In this section, after establishing two preliminary results, we prove Proposition \ref{thm2}. First, we obtain a useful explicit formula relating zeros in short intervals to a Dirichlet polynomial over primes.  A version of this explicit formula first appeared in a paper of Montgomery and Odlyzko \cite{MO}, and later in the work of Gonek \cite{Gonek} and Radziwi\l\l \, \cite{R14}. We require a version with slightly more precise error term for our application.

\begin{lemma} \label{lemma1} Assume RH. Then for $\tau \ge 2$ we have
\begin{equation}\label{explicit}
\begin{aligned}
\sum_\gamma\Big(\frac{\sin\frac12(\gamma-\tau)\log x}{\frac12(\gamma-\tau)\log x}\Big)^{\!2}
&= \frac{\log\frac{\tau}{2\pi}}{\log x} - \frac2{\log x}\sum_{n\le x}\frac{\Lambda(n)}{n^{1/2}}\Big(1-\frac{\log n}{\log x}\Big)\cos (\tau\log n) \\
&\qquad +O\Big(\frac1{\tau\log x}\Big) +O\Big(\frac{x^{1/2}}{(\tau\log x)^2}\Big).
\end{aligned}
\end{equation}
\end{lemma}
\begin{proof}
Let $\alpha(s) = \sum_{n\geq1} a_n \, n^{-s}$ be a Dirichlet series with abscissa of convergence $\sigma_c$.  If $a > \max(0, \sigma_c)$, then
\[
\sum_{n\le x}a_n\log\frac xn = \frac1{2\pi i}\int_{a-i\infty}^{a+i\infty} \alpha(s)\frac{x^s}{s^2}\,ds
\]
for $x > 0$. By three applications of this we deduce that
\[
\sum_{n\le x}a_n\log\frac xn = \frac1{2\pi i}\int_{a-i\infty}^{a+i\infty} \alpha(s)\frac{\big(x^{s/2}-x^{-s/2}\big)^2}{s^2}\,ds
\]
for $x > 1$. We take $\displaystyle \alpha(s) = -\,\frac{\zeta'}{\zeta}\Big(s+\frac12+i\tau\Big)$ and see that
\begin{align}
\sum_{n\le x}\frac{\Lambda(n)}{n^{\frac12+i\tau}}\log\frac xn &=- \frac{1}{2\pi i}\int_{a-i\infty}^{a+i\infty} \frac{\zeta'}{\zeta}\big(s+{\textstyle\frac12}+i\tau\big)
\frac{\big(x^{s/2}-x^{-s/2}\big)^2}{s^2}\,ds\,. \notag
\intertext{We move the contour of integration to the imaginary axis.  In doing so, we encounter a pole at $s = \frac12-i\tau$. On the imaginary axis
we form semi-circular paths around the poles at the points $s=i(\gamma-\tau)$.  We shrink the radii of these semicircles to zero, which gives half-residues,
and the integral on the imaginary axis is defined by taking the Cauchy Principal Value at the point $s = i(\gamma-\tau)$.  Thus, the above is}
\sum_{n\le x}\frac{\Lambda(n)}{n^{\frac12+i\tau}}\log\frac xn &= -\frac{2}{\pi}\int_{-\infty}^\infty \frac{\zeta'}{\zeta}({\textstyle\frac12}+it+i\tau)\Big(\frac{\sin\frac 12t\log x}{t}\Big)^2\,dt
\label{E:Int1} \\
& \qquad -\,\frac12\sum_\gamma \Big(\frac{\sin\frac12(\gamma-\tau)\log x}{\frac12(\gamma-\tau)}\Big)^2
+\Big(\frac{x^{\frac12(\frac12-i\tau)}-x^{-\,\frac12(\frac12-i\tau)}}{\frac12-i\tau}\Big)^2. \notag
\end{align}
 Recall from \cite{GGM} that, for ${\textstyle\frac12}\le \sigma \le 2$, $t\ge 2$, and $s\neq {\textstyle\frac12}+i\gamma $,
\[\Re \frac{\zeta'}{\zeta}\big(\sigma+it\big) = -\,{\textstyle\frac12}\log\frac{t}{2\pi} + \sum_\gamma \frac{\sigma-\frac12}{ (\sigma -\frac12)^2 + (t-\gamma)^2} +O(1/t)
 \]
 and therefore
\[
\Re \frac{\zeta'}{\zeta}\big({\textstyle\frac12}+it\big) = -\,{\textstyle\frac12}\log\frac{t}{2\pi} +O(1/t)
\]
for $t\ge 2$ and $t\neq \gamma$. We take real parts of both sides of \eqref{E:Int1}. An easy calculation shows that the real part of the integral in \eqref{E:Int1} is $= \frac12(\log x)\log\frac{\tau}{2\pi} +O((\log x)/\tau)$ for $\tau\ge 2$. Multiplying both sides of \eqref{E:Int1}
 by $2/(\log x)^2$ and then rearranging, we see that
\begin{equation}\label{E:ExpForm}
\begin{aligned}
\sum_\rho\Big(\frac{\sin\frac12(\gamma-\tau)\log x}{\frac12(\gamma-\tau)\log x}\Big)^{\!2}
&= \frac{\log\frac{\tau}{2\pi}}{\log x} - \frac2{\log x}\sum_{n\le x}\frac{\Lambda(n)}{n^{1/2}}\Big(1-\frac{\log n}{\log x}\Big)\cos (\tau\log n) \\
&\qquad +O\Big(\frac1{\tau\log x}\Big) +O\Big(\frac{x^{1/2}}{(\tau\log x)^2}\Big).
\end{aligned}
\end{equation}
This completes the proof of the lemma.
\end{proof}

\begin{lemma}\label{lemma2}
Assume RH. Let $T\ge 2$ be given. If $k \in \mathbb{N}$, then there is a positive absolute constant $C$ such that \begin{equation} \label{integraln(t,k)}
\int_0^T n(t,k)^{2k}\,dt < (Ck)^{2k} \, T\,.
\end{equation}
\end{lemma}

\begin{proof}
Since $(\sin x)/x$ is monotonically decreasing for $0\le x\le \pi$, it follows that the summand in the sum
on the left hand side of \eqref{E:ExpForm} is $\ge 4/\pi^2$ if $|\gamma - \tau|\le \pi/\log x$. Let $k$ be a positive integer, and set $x = T^{1/k}\ge 2$.
Then, for $0\le \tau \le T$, from Lemma \ref{lemma1} we see that
\begin{equation}\label{E:Est1}
N\Big(\tau+\frac{\pi k}{\log T}\Big)-N\Big(\tau-\frac{\pi k}{\log T}\Big)
\ll k + \frac{k}{\log T}\Big|\sum_{n\le x}\frac{\Lambda(n)}{n^{1/2}}\Big(1-\frac{\log n}{\log x}\Big)\cos (\tau\log n)\Big|+ \frac{T^{1/2k}}{\tau +2}\,.
\end{equation}
By \eqref{n(t,lambda)} we see the left-hand side of \eqref{E:Est1} is $n(t, k)$ with $t= \tau - \pi k/\log T$.
We raise both sides of \eqref{E:Est1} to the power $2k$, and integrate. Since $|\Re z|\le |z|$ for any complex
number $z$, it follows that
\begin{equation}\label{E:Est2}
\int_0^T n(t,k)^{2k}\,dt < (Ck)^{2k} \, T
+ \Big(\frac{Ck}{\log T}\Big)^{\!2k}\int_0^T\Big|\sum_{n\le x}\frac{\Lambda(n)}{n^{1/2+it}}\Big(1-\frac{\log n}{\log x}\Big)\Big|^{2k}\,dt\,.
\end{equation}
Here $C$ is a suitable absolute constant which may be increased appropriately between steps in the argument.

\smallskip

For any real number $b>0$, we have $|f(t)+g(t)|^b \le 2^b \, (|f(t)|^b + |g(t)|^b)$, and therefore we can divide the sum on the right-hand side into the sum over primes plus the sum over prime powers and estimate them separately. For the sum over prime powers, the series over the prime powers $p^\ell$ with $\ell>2$ converges absolutely and makes the contribution $< (Ck/\log T)^{2k} \,T$. The sum over squares of primes is trivially $\ll \log x$ which makes a contribution $\ll C^{2k} \, T$. Thus both of these errors are covered by the first error term on the right-hand side of \eqref{E:Est2}.

\smallskip

For the contribution coming from primes, put $R = \pi(x)$, and let $p_1, p_2, \ldots, p_R$ denote the first $R$ primes. By the multinomial theorem,
\begin{align*}
\Big(\sum_{p\le x}\frac{\log p}{p^{1/2+it}}\Big(1-\frac{\log p}{\log x}\Big)\Big)^k
&= \sum_{\substack{\mu_1, \mu_2, \ldots, \mu_R\\ \sum_r \mu_r = k}}\binom{k}{\mu_1\ \mu_2\ \cdots\ \mu_R}
\prod_{r=1}^R\Big(\frac{\log p_r}{p_r^{1/2+it}}\Big(1-\frac{\log p_r}{\log x}\Big)\Big)^{\!\mu_r} \\
&= \sum_n c(n) \, n^{-it}.
\end{align*}
Here $n = \prod_r p_r^{\mu_r}$, and $c(n) = 0$ if $n > x^k = T$.  By the usual approximate Parseval identity for Dirichlet polynomials \cite[Corollary 3]{MV}, we see that
\begin{align*}
\int_0^T\!\Big|\sum_{p\le x}\frac{\log p}{p^{1/2+it}}\Big(1-\frac{\log p}{\log x}\Big)\Big|^{2k}dt
&\ll T\sum_n |c(n)|^2 \\
&= T\!\sum_{\substack{\mu_1, \mu_2, \ldots, \mu_R\\ \sum_r \mu_r = k}}\!\!\binom{k}{\mu_1\ \mu_2\ \cdots\ \mu_R}^{\!2}\!
\prod_{r=1}^R\Big(\frac{\log p_r}{p_r^{1/2}}\Big(1-\frac{\log p_r}{\log x}\Big)\Big)^{\!2\mu_r}.
\intertext{The multinomial coefficient above is $\le k!$ for any collection of $\mu_r$, so the above is}
&\le k! \, T\!\sum_{\substack{\mu_1, \mu_2, \ldots, \mu_R\\ \sum \mu_r = k}}\!\!\binom{k}{\mu_1\ \mu_2\ \cdots\ \mu_R}\!
\prod_{r=1}^R\Big(\frac{\log p_r}{p_r^{1/2}}\Big(1-\frac{\log p_r}{\log x}\Big)\Big)^{\!2\mu_r},
\intertext{which by the multinomial theorem is}
&=k! \, T\bigg(\sum_{p\le x}\frac{\log^2 p}p \, \Big(1-\frac{\log p}{\log x}\Big)^2\bigg)^{\!k} \\
&\sim k! \, T \, \Big(\frac{ \log^2 x}{12}\Big)^{\!k} < \Big(\frac Ck\Big)^kT \, (\log T)^{2k}.
\end{align*}
Hence the contribution from primes in the second term on the right-hand side of \eqref{E:Est2}  is $< (Ck)^k \, T$, which for large $k$ is smaller than the contribution of the first term. This completes the proof.
\end{proof}

{\sc Remarks.} (1) It is trivial that
\[
\int_0^T n(t,k)\,dt \sim kT
\]
as $T\to\infty$ with $k$ fixed.

\smallskip

(2) A form of Lemma 2 has been proved unconditionally by Fujii \cite{Fujii81}, and he also proved there unconditionally a form of the multiplicity bound \eqref{multiplicitymoment} in Proposition \ref{thm2}.

\smallskip

(3) The second term on the right hand side of \eqref{E:Est2} contributes less than the first term, which suggests the possibility that
\[
\int_0^T (n(t,k) - k)^{2k}\,dt < (Ck)^k \, T\,.
\]
It might be interesting, and perhaps even valuable in some contexts, to establish such a result.
\smallskip

(4) Our estimate for the $2k$-th moment for a Dirichlet polynomial over primes in the above proof is essentially the same as the proof of Lemma 3 of Soundararajan \cite{So09}. The use of the approximate Parseval identity for
Dirichlet polynomials in our argument allows us to replace the condition $x^k \le T/\log T$ appearing in \cite{So09} by $x^k \ll T$.

\medskip

We are now in a position to prove Proposition \ref{thm2}.

\begin{proof}[Proof of Proposition \ref{thm2}]
For $t \in \mathbb{R}$ with $0<t\le T$, we put $I(t) = [t, t+2\pi/\log T]$. We note that if $\gamma\in I(t)$, then $(\gamma, \gamma+2\pi k/\log T] \subset (t, t+2\pi(k+1)/\log T]$,
and hence for such $\gamma$ we have $n(\gamma,k)\le n(t,k+1)$.  For ordinates $\gamma$ with $0 < \gamma \le T$, we put $J(\gamma) = [\gamma -2\pi/\log T,\gamma]$
and we note that $\gamma \in I(t)$ if and only if $t\in J(\gamma)$.  Thus $n(\gamma,k)^{2k}\le n(t,k+1)^{2k}$ for $t\in J(\gamma)$. We average over these $t$ to see
that
\[
n(\gamma,k)^{2k}\le \frac{\log T}{2\pi}\int_{J(\gamma)}n(t,k+1)^{2k}\,dt\,.
\]
Summing over $\gamma$, we see that
\[
\sum_{0<\gamma \le T}n(\gamma,k)^{2k}\le \frac{\log T}{2\pi}\sum_{\gamma\le T}\int_{J(\gamma)}n(t,k+1)^{2k}\,dt
= \frac{\log T}{2\pi}\int_0^T n(t,k+1)^{2k}\Big(\sum_{\gamma\in I(t)} 1\Big)\,dt\,.
\]
The number of $\gamma \in I(t)$ is $n(t,1)$ unless $t$ is the ordinate of a zero, which occurs at only finitely many points in $(0,T]$.  Thus, the right-hand side
above is
\[
= \frac{\log T}{2\pi}\int_0^T n(t,k+1)^{2k} \, n(t,1)\,dt \le \frac{\log T}{2\pi}\int_0^T n(t,k+1)^{2k+1}\,dt\,.
\]
Now $n(t,k+1)$ is a nonnegative integer, so $n(t,k+1)^{2k+1}\le n(t,k+1)^{2k+2}$ for all $t$. Hence, by Lemma \ref{lemma2}, the right-hand side above is
\[
\le \frac{\log T}{2\pi}\int_0^T n(t,k+1)^{2k+2}\,dt \ll (Ck)^{2k} \, T\log T.
\]
This establishes the desired bound in \eqref{zeromoment}.

 \smallskip

To prove \eqref{multiplicitymoment}, observe that
\[
n(t,k)^{2k} = \bigg(\sum_{ t\le \gamma_d\le t + 2\pi k/\log T} m(\gamma_d)\bigg)^{\!2k}
\ge \sum_{ t\le \gamma_d\le t + 2\pi k/\log T} m(\gamma_d)^{2k}.
\]
Integrating both sides of this inequality and applying Lemma \ref{lemma2}, it follows that
\[
\int_0^T \sum_{ t\le \gamma_d\le t + 2\pi k/\log T} m(\gamma_d)^{2k} \,dt
\ \le \ \int_0^T n(t,k)^{2k}\,dt \ < \ (Ck)^{2k} \, T.
\]
Moreover, the left-hand side above is
\[
\ge \sum_{ 0<\gamma_d \le T } m(\gamma_d)^{2k}\int_{\gamma_d-2\pi k/\log T}^{\gamma_d} 1\,dt
\ = \ \frac{2\pi k}{\log T} \sum_{ 0<\gamma_d \le T} m(\gamma_d)^{2k}.
\]
Combining estimates gives the bound in \eqref{multiplicitymoment}. This completes the proof of the proposition.
\end{proof}

\smallskip

There is considerable overlap between \eqref{zeromoment} in Proposition \ref{thm2} and Lemma 10 of Radziwi{\l}{\l} \cite{R14}. Radziwi{\l}{\l} proves in that paper a relationship between small gaps between zeros of $\zeta(s)$ and zeros of $\zeta'(s)$ close to the half-line. Assuming RH, he proves a form of Lemma 1 and then uses a moment argument based on the Gonek-Landau formula \cite{Gonek}. In contrast, as mentioned above, our proof of Lemma 2 is similar to the argument used by Soundararajan in \cite[Lemma 3]{So09}.

\section{Proofs of Theorem \ref{thm3} and Corollary \ref{cor1} } \label{sec4}

We now deduce Theorem \ref{thm3} from Propositions \ref{thm1} and \ref{thm2}.

\begin{proof}[Proof of Theorem \ref{thm3}]
Assume RH. For $r \in \mathcal{A}(\lambda)$, suppose that $\lambda > 0$ is a number for which $c(\lambda;r) > 0$. Then, by \eqref{thm1result}, we have
\[
\sum_{0<\gamma_d \le T } m(\gamma_d) \, \big(m(\gamma_d)-1\big) \ \ + \!\!\!\! \sum_{\substack{0<\gamma,\gamma'\le T \\ 0<|\gamma-\gamma'|\le \frac{2\pi\lambda}{\log T}}} \!\!\!\!\!\!\!\! 1  \
\ge \ \frac12 \, c(\lambda;r) \, N(T)\,
\]
for sufficiently large $T$. Consequently, at least one of the following is true:
\begin{equation}\label{E:alt1}
\sum_{0<\gamma_d \le T } m(\gamma_d) \, \big(m(\gamma_d)-1\big) \ \ge \ \frac14 \, c(\lambda;r) \, N(T) \\
\end{equation}
or
\begin{equation}\label{E:alt2}
\sum_{\substack{0<\gamma,\gamma'\le T \\ 0<|\gamma-\gamma'|\le \frac{2\pi\lambda}{\log T}}} \!\!\!\!\!\!\!\! 1 \ \ge \ \frac14 \, c(\lambda;r) \, N(T).
\end{equation}
If \eqref{E:alt1} holds, then by Cauchy's inequality we have
\[
\Big(\frac14 \, c(\lambda;r)\, N(T)\Big)^2 \le \bigg(\sum_{\substack{0<\gamma_d \le T \\ m(\gamma_d)>1}}1\bigg)
\bigg(\sum_{0<\gamma_d \le T} m(\gamma_d)^4\bigg).
\]
By taking $k=2$ in Proposition \ref{thm2}, we see that the second factor on the right-hand side above is $\ll T\log T$. Therefore
\[
  c(\lambda;r)^2 \, N(T) \ll \sum_{\substack{0<\gamma_d \le T\\ m(\gamma_d)>1}}1 \le N(T) \, D(0,T) \le N(T) \, D(\lambda , T) ,
\]
which proves \eqref{thm3:1} in this case.

\smallskip

We now need to consider the case where \eqref{E:alt2} holds. For this, we first need to find an upper bound for the sum on the left-hand side of \eqref{E:alt2}. By definition, $n(t,\lambda)$ counts ordinates (with multiplicity) in the interval $(t, t+2\pi \lambda/\log T]$. Hence $n(\gamma,\lambda)$ counts ordinates $\gamma'$ with multiplicity $m(\gamma')$ such that $\gamma<\gamma'\le \gamma + 2\pi \lambda/\log T$. Thus
\[
\sum_{\substack{0<\gamma,\gamma'\le T \\ 0<|\gamma-\gamma'|\le \frac{2\pi\lambda}{\log T}}} \!\!\!\!\!\!\!\! 1 \ = \ 2 \sum_{0<\gamma \le T} n(\gamma,\lambda) \ = \ 2 \sum_{0<\gamma_d \le T } m(\gamma_d) \, n(\gamma_d,\lambda) .
\]
Two applications of Cauchy's inequality give
\[
\sum_{\substack{0<\gamma,\gamma'\le T \\ 0<|\gamma-\gamma'|\le \frac{2\pi\lambda}{\log T}}} \!\!\!\!\!\!\!\! 1 \ \le \ 2 \,  \bigg( \sum_{0<\gamma_d \le T} m(\gamma_d)^2\bigg)^{1/2} \bigg( \sum_{\substack{0<\gamma_d \le T\\ n(\gamma_d,\lambda)\ge 1}} 1 \bigg)^{1/4}\bigg( \sum_{0<\gamma \le T} n(\gamma,\lambda)^4\bigg)^{1/4} .
\]
Since  $n(\gamma_d,\lambda)\ge 1$ if and only if $\gamma_d^+-\gamma_d \le 2\pi \lambda/\log T$, we conclude from both \eqref{zeromoment} and \eqref{multiplicitymoment} of Proposition \ref{thm2} and the definitions in \eqref{D(lambda)} that
\begin{equation} \label{laststep}
\sum_{\substack{0<\gamma,\gamma'\le T \\ 0<|\gamma-\gamma'|\le \frac{2\pi\lambda}{\log T}}} \!\!\!\!\!\!\!\! 1 \ \ll \ (T\log T) \,   D_d(\lambda, T)^{1/4} \ \le \ (T\log T) \,  D(\lambda, T)^{1/4} .
\end{equation}
Assuming that \eqref{E:alt2} holds, we see from \eqref{laststep} that $D(\lambda,T) \gg c(\lambda;r)^4$  and \eqref{thm3:1} is established.

\smallskip

To prove \eqref{mustar}, we have
\[\sum_{0<\gamma_d \le T} m(\gamma_d)(m(\gamma_d)-1) \le \big(n^*-1 + o(1)\big) \, N(T),
\]
and the result follows immediately from \eqref{thm1result}
and \eqref{laststep}. This completes the proof of the theorem.
\end{proof}

\begin{proof}[Proof of Corollary \ref{cor1}]
To make our calculations easy to verify, we will use the Selberg minorant for the characteristic function of the interval $[-1,1]$. Slightly better results can be obtained by using the more elaborate methods in \cite{CGdL}. Let
\begin{equation}\label{E:DefR}
R(x) = \Big( \frac{\sin \pi x}{\pi x} \Big)^2 \frac{1}{1-x^2} 
\end{equation}
so that $R(0) = 1$ and $R(x)\le 0$ for $|x|\ge 1$. We note that
\begin{equation*}\label{E:Rhat}
\widehat R(t) = \begin{cases} 1 - |t|+\displaystyle \frac{\sin 2\pi|t|}{2\pi}, & \text{if } |t|\le 1, \\
                               0,                               &\text{otherwise}.
                \end{cases}
\end{equation*}
Hence $R\in \mathcal{A}(1)$. For any $\lambda>0$, set $r(u) = R(u/\lambda)$ and note that $r\in \mathcal{A}(\lambda)$ with $\widehat r(\alpha) = \lambda \, \widehat R(\lambda \alpha)$.
Thus
\[ \begin{split}
c(\lambda;r) &= \widehat r(0) - 1 + 2\int_0^1\alpha \, \widehat r(\alpha)\,d\alpha
\\ & = \lambda -1 +2\lambda \int_0^{\min(1, 1/\lambda)} \alpha \, \Big(1 - \lambda \alpha +\frac{\sin 2 \pi\lambda \alpha}{2\pi}\Big) \, d\alpha,
\end{split}
\]
and a straightforward numerical calculation shows that
\[
c(0.60729;r)  > 0\,.
\]
Therefore, by Theorem \ref{thm3}, we conclude that $ \mu_D \le 0.60729$. By comparison, the result $\mu_D \le 0.6039$ in \eqref{PCmu} is obtained from a much more complicated minorant.

\smallskip

To find distinct gaps, we need $c(\lambda;r) > n^*-1$. Using the best known bound $n^* \le 1.3208$ from \cite{CGdL}, we find that $c(1.05214;r)> 0.3208$, which establishes the bound $\mu_{D_d} \le 1.05214$.
\end{proof}

{\sc Remark}. We mention in passing that since
\[
c(1;r) = \frac13 - \frac1{2\pi^2} = 0.282673\ldots,
\]
in order to prove $\mu_{D_d} <1$, we would need a bound of $n^*\le 1.2826$.

\section{Proof of Theorem \ref{thm4}}

We now introduce a new method in order to prove that $\mu_d <1$. Write the functional equation for the zeta function as $\zeta(s) =\chi(s) \, \zeta(1-s)$, where $\chi(s)=2^s\pi^{s-1}\sin(\frac{\pi}{2}s) \Gamma(1-s)$. It is not hard to see that $\chi(s) = \chi(1-s)^{-1}$ and $|\chi(\frac{1}{2}+it)|=1$ for real $t$. The Hardy $Z$-function is defined by
\[
Z(t) = \chi(\tfrac{1}{2}- it)^{1/2} \zeta(\tfrac{1}{2}+ it) = \chi(\tfrac{1}{2}+ it)^{1/2} \zeta(\tfrac{1}{2}- it).
\]
It follows from the functional equation that $Z(t)$ is real for $t\in\mathbb{R}$, that $|Z(t)|=|\zeta(\frac{1}{2}+it)|$, and that $Z(t)$ changes sign when $t$ corresponds to an ordinate of a zero of odd multiplicity of $\zeta(s)$ on the critical line.

\smallskip

Assume RH. If $\gamma \in [T,2T]$ is an ordinate of a simple zero of $\zeta(s)$ and $0<a < \mu_d$, then
\[
Z'(\gamma) \, Z\Big(\gamma+\frac{2\pi a}{\log T}\Big) >0
\]
when $T$ is sufficiently large. This inequality is easily established by considering cases. If $Z'(\gamma)>0$, then $Z\Big(\gamma+\frac{2\pi a}{\log T}\Big) >0$ for $0<a < \mu_d$.  If $Z'(\gamma)<0$, then $Z\Big(\gamma+\frac{2\pi a}{\log T}\Big)<0$ for $0<a < \mu_d$. Finally, if $Z'(\gamma)=0$, then $\gamma$ corresponds to a multiple zero of $\zeta(s)$.
Likewise, under the same assumptions,
\[
Z'(\gamma) \, Z\Big(\gamma+\frac{2\pi a}{\log T}\Big) \, \big|f (\gamma)\big|^2 \geq 0
\]
for any function $f$. Consequently, if there exist choices of $\kappa$ and $f$ such that
\[
\sum_{T<\gamma \le 2T}Z'(\gamma) \, Z\Big(\gamma+\frac{2\pi \kappa}{\log T}\Big) \, \big|f (\gamma)\big|^2 < 0
\]
when $T$ is sufficiently large, then it follows that $\mu_d \le \kappa$. Note that this method is detecting spacings between a simple zeros of $Z(t)$ and a sign changes of $Z(t)$ located at another location. In other words, it is detecting (nonzero) gaps between the ordinates of a simple zeros of $\zeta(s)$ and ordinates of zeros of $\zeta(s)$ of odd multiplicity. We could alternatively study mean-values of the form
\[
\sum_{T<\gamma \le 2T}Z'(\gamma) \, Z\Big(\gamma-\frac{2\pi \kappa}{\log T}\Big) \, \big|f (\gamma)\big|^2
\]
in order to detect gaps between zeros of odd multiplicity followed by simple zeros.

\smallskip

As in Selberg’s proof that a positive proportion of the zeros of $\zeta(s)$ are on the critical line, we are studying
sign changes of the Hardy $Z$-function on the scale of average spacing between zeros. This intuition suggests that we should choose the test function $f$ to mollify the product $Z'(\gamma) \, Z\Big(\gamma+\frac{2\pi a}{\log T}\Big)$. It is convenient to choose a mollifier of the form
\begin{equation*}
M(s,P)=\sum_{n\leq y}\frac{\mu(n)P\big(\frac{\log y/n}{\log y}\big)}{n^s},
\end{equation*}
where $\mu(n)$ is the M\"obius function, $y=T^{\vartheta}$, and $P(x)$ a polynomial satisfying $P(0)=0$.
Thus, if there exist choices of $\kappa, \eta, \vartheta$, and $P$, such that
\[
\Sigma(\kappa;\eta,P) :=\sum_{T<\gamma \le 2T}Z'(\gamma) \, Z\Big(\gamma+\frac{2\pi\kappa}{\log T}\Big)\, \Big|M\Big(\frac12+i\gamma+\frac{2\pi i\eta}{\log T},P\Big)\Big|^2 < 0
\]
when $T$ is sufficiently large, then it follows that $\mu_d \le \kappa$. We choose $0 < \eta < \kappa$, so that we are simultaneously mollifying both $Z'(\gamma)$ and $Z\Big(\gamma+\frac{2\pi a}{\log T}\Big)$. This mean-value can be analyzed using techniques similar to those introduced by Conrey, Ghosh, and Gonek \cite{CGG1} 
and Bui and Heath-Brown \cite{BH-B} who were interested in estimating the proportion of simple zeros of the zeta function.

\smallskip

In particular, let
\[
M(s,g)=\sum_{n\leq y}\frac{\mu(n) \, g(\frac{\log y/n}{\log y})}{n^s},
\]
where $y=T^\vartheta$ and $g$ is entire with $g(0)=0$. Let $Q_1$, $Q_2$ be polynomials and let
\begin{align*}
&I(a,b,g_1,g_2,Q_1,Q_2)\\
&\qquad=\sum_{T<\gamma \leq 2T}Q_1\Big(-\frac{d}{da}\Big) \zeta\Big(\rho+\frac{a}{\log T}\Big) \, Q_2\Big(-\frac{d}{db}\Big)\zeta\Big(1-\rho+\frac{b}{\log T}\Big)\,M(\rho,g_1) \, M(1-\rho,g_2).
\end{align*}
Then the following estimate holds.

\begin{theorem}[Conrey--Ghosh--Gonek]\label{mainthm2}
	If $0<\vartheta<1/2$ and $a,b\ll 1$, then as $T\rightarrow\infty$ we have
	\begin{align*}
		&I(a,b,g_1,g_2,Q_1,Q_2)\sim\frac{T\log T}{2\pi}\frac{\partial^2}{\partial u \, \partial v}\Bigg\{\bigg(\frac{1}{\vartheta}\int_{0}^{1}g_1\, g_2\, dx+\int_{0}^{1}g_1\, dx\int_{0}^{1}g_2\, dx\bigg)\\
		&\qquad\qquad\qquad\qquad\qquad\qquad\qquad\qquad\qquad\times\bigg(\int_{0}^{1}T_aQ_1\, T_bQ_2\, dx-\int_{0}^{1}T_aQ_1\, dx\int_{0}^{1}T_bQ_2\, dx\bigg)\\
		&\qquad\qquad\qquad+\int_{0}^{1}g_1\, dx\int_{0}^{1}g_2\, dx\bigg(Q_1(0)-\int_{0}^{1}T_aQ_1\, dx\bigg)\bigg(Q_2(0)-\int_{0}^{1}T_bQ_2\, dx\bigg)\Bigg\}\Bigg|_{u=v=0},
	\end{align*}
where $g_1=g_1(x+u)$, $g_2=g_2(x+v)$,
\[
T_aQ_1=e^{-a(x+\vartheta u)}Q_1(x+\vartheta u), \quad\text{and}\quad T_bQ_2=e^{-b(x+\vartheta v)}Q_2(x+\vartheta v).
\]
\end{theorem}

This result was stated without a proof by Conrey, Ghosh, and Gonek in \cite[Theorem 2]{CGG87}. It can  be deduced in a straightforward manner from the more general result of Heap, Li, and Zhao \cite[Theorem 2]{HLZ} which builds on the previous work in \cite{CGG1} and  \cite{BH-B}.

\smallskip

\begin{proof}[Proof of Theorem \ref{thm4}]
Using Stirling's formula, it is not difficult to convert the mean-value $\Sigma(\kappa;\eta,P)$ into a mean-value of the form that is estimated in Theorem \ref{mainthm2}. We start by observing that if $\rho=\frac{1}{2}+i\gamma$ is a zero of $\zeta(s)$, then it follows from the definition of $Z(t)$ that
\[
Z'(\gamma) = -i\,\chi(\rho)^{1/2} \, \zeta'(1-\rho).
\]
Hence
\begin{align*}
Z'(\gamma) \, Z\Big( \gamma+\frac{2 \pi \kappa}{\log T} \Big)=-i\,\chi(\rho)^{1/2} \, \chi\Big(1-\rho -\frac{2\pi i \kappa}{\log T}\Big)^{1/2} \, \zeta'(1-\rho) \,  \zeta\Big(\rho+\frac{2\pi i \kappa}{\log T}\Big).
\end{align*}
Using the Stirling's formula approximation
\[
\chi(s+a) \, \chi(1-s+b)= \Big(\frac{t}{2\pi}\Big)^{-a-b}\bigg(1+O\Big(\frac{1}{1+|t|}\Big)\bigg)
\]
for $t$ large and $a,b$ uniformly bounded, we have
\begin{align*}
\chi(\rho)^{1/2}\, \chi\Big(1-\rho -\frac{2\pi i \kappa}{\log T}\Big)^{1/2}&= \Big(\frac{\gamma}{2\pi}\Big)^{\pi i \kappa/\log T}\bigg(1+O\Big(\frac{1}{T}\Big)\bigg)\\
&= e^{\pi i \kappa} \, \Big( 1 + O\big((\log T)^{-1} \big) \Big)
\end{align*}
for $\gamma \in (T,2T]$. Therefore
\begin{align}\label{sigma}
\Sigma(\kappa;\eta,P) =&\, -i e^{\pi i \kappa} \sum_{T<\gamma \le 2T} \zeta'(1-\rho) \, \zeta\Big(\rho+\frac{2\pi i \kappa}{\log T}\Big)\, M\Big(1-\rho-\frac{2\pi i\eta}{\log T},P\Big) \, M\Big(\rho+\frac{2\pi i\eta}{\log T},P\Big)\nonumber\\
&\qquad+O\bigg((\log T)^{-1}\sum_{T<\gamma \le 2T} \Big|\zeta'(1-\rho) \, \zeta\Big(\rho+\frac{2\pi i \kappa}{\log T}\Big)\Big|\, \Big|M\Big(\rho+\frac{2\pi i\eta}{\log T},P\Big)\Big|^2\bigg).
\end{align}
Using Theorem \ref{mainthm2} and Cauchy's integral formula, we can deduce the bounds
\[
 \sum_{T<\gamma \le 2T} \Big|\zeta\Big(\rho+\frac{2\pi i \kappa}{\log T}\Big)\, M\Big(\rho+\frac{2\pi i\eta}{\log T},P\Big)\Big|^2\ll T\log T
\]
and
\[
 \sum_{T<\gamma \le 2T} \Big|\zeta'(\rho)\, M\Big(\rho+\frac{2\pi i\eta}{\log T},P\Big)\Big|^2\ll T \log^3 T
\]
for any fixed $\kappa$, $\eta$, and $P$. Hence, using Cauchy's inequality and these bounds to estimate the error term in \eqref{sigma}, it follows from Theorem \ref{mainthm2} that
\begin{align*}
\Sigma(\kappa;\eta,P)& =-i e^{\pi i \kappa} \sum_{T<\gamma \le 2T} \zeta'(1-\rho) \, \zeta\Big(\rho+\frac{2\pi i \kappa}{\log T}\Big)\, M\Big(1-\rho-\frac{2\pi i\eta}{\log T},P\Big) \, M\Big(\rho+\frac{2\pi i\eta}{\log T},P\Big)\\
&\qquad +O\big(T\log T\big)
\\
&= -i e^{\pi i \kappa} \, (\log T) \, I(2\pi i\kappa,0,g_1,g_2,1,-x)+O\big(T\log T)
\end{align*}
with
\[
g_1(x)=P(x)\exp\big(-2\pi i\vartheta\eta(1-x)\big)\quad\text{and}\quad g_2(x)=P(x)\exp\big(2\pi i\vartheta\eta(1-x)\big).
\]
The choice $\vartheta=0.4999$, $\kappa=0.991$, $\eta=0.6$, and $P(x)=x$ gives
\[
\Sigma(0.991;0.6,x) = \Big(-0.00155\ldots+o(1)\Big) \, \frac{T \log^2 T}{2\pi},
\]
and hence Theorem \ref{thm4} follows.
\end{proof}

{\sc Remark}. We have not tried to find the smallest admissible value of $\kappa$ in the above argument, and instead focussed on finding simple choices of $\eta$ and $P$ that yield a value of $\kappa<1$.

\section{Proof of Corollary \ref{corollaryquantitativeresult}}

Let $T$ is large and let $\vartheta, \kappa, \eta$, and $P$ be chosen as in the proof of Theorem \ref{thm4}. Then
\[
\sum_{T<\gamma \le 2T} Z'(\gamma) \, Z\Big(\gamma+\frac{2\pi\kappa}{\log T}\Big)\, \Big|M\Big(\frac12+i\gamma+\frac{2\pi i\eta}{\log T},P\Big)\Big|^2 < 0.
\]
Observe that every negative term in this sum corresponds to zero with $Z'(\gamma) \, Z\big(\gamma+\frac{2\pi\kappa}{\log T}\big) <0$, however not necessarily every term in this sum is  negative. Moreover, Theorem \ref{mainthm2} shows that this sum is $\gg T \log^2T$ in magnitude. Therefore,
\[
\begin{split}
T \log^2T &\ll \bigg| \sum_{T<\gamma \le 2T} Z'(\gamma) \, Z\Big(\gamma+\frac{2\pi\kappa}{\log T}\Big)\, \Big|M\Big(\frac12+i\gamma+\frac{2\pi i\eta}{\log T},P\Big)\Big|^2 \bigg|
\\
&\le
\bigg| \sum_{\substack{T<\gamma \le 2T \\ Z'(\gamma) \, Z(\gamma+\frac{2\pi\kappa}{\log T}) <0  }} Z'(\gamma) \, Z\Big(\gamma+\frac{2\pi\kappa}{\log T}\Big)\, \Big|M\Big(\frac12+i\gamma+\frac{2\pi i\eta}{\log T},P\Big)\Big|^2 \bigg|.
\end{split}
\]
Applying Cauchy's inequality to the latter sum, we deduce that
\[
\begin{split}
T^2 \log^4T \le \bigg( \sum_{\substack{T<\gamma \le 2T \\ Z'(\gamma) \, Z(\gamma+\frac{2\pi\kappa}{\log T}) <0  }} \!\!\!\! 1 \bigg)  \bigg( \sum_{T<\gamma \le 2T} Z'(\gamma)^2 \, Z\Big(\gamma+\frac{2\pi\kappa}{\log T}\Big)^2\, \Big|M\Big(\frac12+i\gamma+\frac{2\pi i\eta}{\log T},P\Big)\Big|^4 \bigg),
\end{split}
\]
where (by positivity) we have extended the second sum on the right-hand side to all zeros with $T<\gamma\le 2T$. By \cite{CC} and \cite{ChandeeSound}, for $T<\gamma\le 2T$, we have
\[
Z'(\gamma)^2 \, Z\Big(\gamma+\frac{2\pi\kappa}{\log T}\Big)^2 \ll \exp\!\bigg( \big(\log 4 + o(1)\big) \frac{\log T}{\log\log T} \bigg).
\]
An upper bound for $|\zeta(\frac{1}{2}+it)|$ is given in \cite{ChandeeSound} whereas the results in \cite{CC} combined a standard argument using the functional equation, Stirling's formula, and Cauchy's integral formula can be used to show that a bound of similar strength holds for $|\zeta'(\frac{1}{2}+it)|$ and $|Z'(t)|$.

\smallskip

Now for $\{a(n)\} \subseteq \mathbb{C}$ with $|a(n)|\ll n^\varepsilon$ and $x \le T^{1-\varepsilon}$ for $\varepsilon>0$, it is known (e.g.~using the Landau--Gonek explicit formula \cite[Theorem 1]{Gonek} or contour integration \cite[Theorem 5.1]{Li}) that
\[
\sum_{0<\gamma\le T} \bigg| \sum_{n\le x} \frac{a(n)}{n^{\frac{1}{2}+i\gamma}} \bigg|^2 \sim N(T) \sum_{n\le x} \frac{|a(n)|^2}{n} - \frac{T}{\pi} \mathrm{Re} \sum_{n\le x} \frac{(\Lambda*a)(n) \, \overline{a(n)}}{n},
\]
as $T\to \infty$. From this, letting $d(n)$ denote the number of divisors of $n$, it follows that
\[
\begin{split}
\sum_{T<\gamma\le 2T} \bigg| M\Big( \frac{1}{2}+i\gamma + \frac{2\pi i \kappa}{\log T},P \Big) \bigg|^4 &\ll T\log T \sum_{n\le y^2} \frac{d(n)^2}{n} + T \sum_{n\le y^2} \frac{(\Lambda*d)(n) \, d(n)}{n}
\\
&\ll T \log^5 T
\end{split}
\]
for our choice of $M(s,P)$ and $y$ in the proof of Theorem \ref{thm4}. Hence
\[
\sum_{T<\gamma \le 2T} Z'(\gamma)^2 \, Z\Big(\gamma+\frac{2\pi\kappa}{\log T}\Big)^2\, \Big|M\Big(\frac12+i\gamma+\frac{2\pi i\eta}{\log T},P\Big)\Big|^4 \ll T \, \exp\!\bigg( \big(\log 4 + o(1)\big) \frac{\log T}{\log\log T} \bigg).
\]
Combining estimates, we conclude that
\[
\sum_{\substack{T<\gamma \le 2T \\ Z'(\gamma) \, Z(\gamma+\frac{2\pi\kappa}{\log T}) <0  }} \!\!\!\!\!\! 1 \ \gg \ T \, \exp\!\bigg( -\big(\log 4 + o(1)\big) \frac{\log T}{\log\log T} \bigg).
\]
In this way, we deduce Corollary \ref{corollaryquantitativeresult} from the proof of Theorem \ref{thm4}.

\medskip

\noindent{\sc Acknowledgements}. We thank Andr\'es Chirre, Brian Conrey, and Caroline Turnage-Butterbaugh for  some helpful comments and suggestions. MBM~was supported by the NSF grant DMS-2101912 and the Simons Foundation (award 712898).


\begin{thebibliography}{9}



\bibitem{BH-B}\label{BH-B}
H. M. Bui, D. R. Heath-Brown, \textit{On simple zeros of the Riemann zeta-function}, Bull. London Math. Soc. \textbf{45} (2013), 953--961.

\bibitem{BM}
H. M. Bui, M. B. Milinovich, {\it Gaps between zeros of the Riemann zeta-function}, Q. J. Math. {\bf 69} (2018), no. 2, 403--423.

\bibitem{BMN}\label{BMN}
H. M. Bui, M. B. Milinovich, Nathan Ng, \textit{A note on the gaps between consecutive zeros of the Riemann zeta-function}, Proc. Amer. Math. Soc. \textbf{138} (2010), 4167--4175.

\bibitem{CC} E.~Carneiro, V.~Chandee, {\it Bounding $\zeta(s)$ in the critical strip}, J. Number Theory {\bf 131} (2011), no. 3, 363--384.

\bibitem{CCLM} E.~Carneiro, V.~Chandee, F.~Littmann, M.~B.~Milinovich, {\it Hilbert spaces and the pair correlation of zeros of the Riemann zeta-function}, J. Reine Angew. Math. {\bf 725} (2017), 143--182

\bibitem{ChandeeSound}
V.~Chandee, K.~Soundararajan, {\it Bounding $|\zeta(\frac{1}{2}+it)|$ on the Riemann hypothesis},  Bull. Lond. Math. Soc. \textbf{43} (2011), 243--250.
 \bibitem{CGdL} \label{CGdL} A. Chirre, F. Gon\c{c}alves, D. de Laat, {\it Pair correlation estimates for the zeros of the zeta function via semidefinite programming}, Adv. Math. {\bf 361} (2020), 106926.


\bibitem{CGG1}\label{CGG1}
J. B. Conrey, A. Ghosh, S. M. Gonek, \textit{Simple zeros of the Riemann zeta function}, Proc. London Math. Soc. \textbf{76} (1998), 497--522.

\bibitem{CGG87}\label{CGG87}
J. B. Conrey, A. Ghosh, S. M. Gonek, \textit{Mean values of the Riemann zeta-function with application to the distribution of zeros}, Number theory, trace formulas and discrete groups (Oslo, 1987), Academic Press, Boston MA, 1989, 185--199.

\bibitem{CGG84}\label{CGG84}
J. B. Conrey, A. Ghosh, S. M. Gonek, \textit{A note on gaps between zeros of the zeta function}, Bull. London Math. Soc. \textbf{16} (1984), 421--424.

\bibitem{CGGGHB} \label{CGGGHB}
J. B. Conrey, A. Ghosh,  D. Goldston, S. M. Gonek, D. R. Heath-Brown, {\it On the distribution of gaps between zeros of the zeta-function}, Q. J. Math. {\bf 36} (1985), 43--51.

\bibitem{CI}
J. B. Conrey, H. Iwaniec, {\it Spacing of zeros of Hecke L-functions and the class number problem}, Acta Arith. {\bf 103} (2002), no. 3, 259--312.


 \bibitem{CT-B}\label{CT-B} J. B. Conrey, C. L. Turnage-Butterbaugh, {\it On $r$-gaps between zeros of the Riemann zeta-function}, Bull. London Math. Soc. \textbf{50} (2018), 349--356.

\bibitem{FW} \label{FW} S.~Feng, X.~Wu, {\it On gaps between zeros of the Riemann zeta-function}, J. Number Theory {\bf 132} (2012), no. 7, 1385--1397


\bibitem{Fujii75} \label{Fujii75} A. Fujii, \textit{On the distribution of the zeros of the Riemann zeta function in short intervals}, Bull. Amer. Math. Soc. {\bf 81} (1975), 139--142.

\bibitem{Fujii81} \label{Fujii81} A. Fujii, \textit{On the zeros of Dirichlet L-functions. II}, Trans. Amer. Math. Soc. {\bf 267} (1981), 33--40.

\bibitem{GGM}\label{GGM} D. A. Goldston, S. M. Gonek, H. L. Montgomery, {\it Mean values of the logarithmic derivative of the Riemann zeta-function with applications to primes in short intervals}, J. reine angew. Math. {\bf 537} (2001), 105--126.

\bibitem{GGOS}
D.~A.~Goldston, S.~M.~Gonek, A.~E.~\~Ozl\"uk, C.~Snyder, {\it On the pair correlation of zeros of the Riemann zeta-function}, Proc. Lond. Math. Soc. (3) {\bf 80} (2000), no. 1, 31--49.


\bibitem{GoldMont}\label{GM} D. A. Goldston, H. L. Montgomery, {\it Pair correlation of zeros and primes in short intervals,} Analytic Number Theory and Diophantine Problems, Birkhauser
Verlag, 1987, 183--203.

\bibitem{GTTB}
D.~A.~Goldston, T.~S.~Trudgian, C.~L.~Turnage-Butterbaugh, {\it A limitation on proving the existence of small gaps between zeta-zeros}, preprint: arXiv:2201.10676.

\bibitem{Gonek}
S.~M.~Gonek,
\newblock {\it An explicit formula of Landau and its applications to the theory of the zeta-function},
\newblock A tribute to Emil Grosswald: number theory and related analysis, Contemp. Math., 143, Amer.~Math.~Soc., Providence, RI, 1993, 395--413.

\bibitem{HLZ}\label{HLZ}
W. Heap, J. Li, J. Zhao, {\it Lower bounds for discrete negative moments of the Riemann zeta function}, Algebra and Number Theory, to appear. Preprint at arXiv:2003.09368.

\bibitem{Li}
J. Li, {\it Large values of Dirichlet L-functions at zeros of a class of L-functions}, Canad. J. Math. {\bf 73} (2021), no. 6, 1459–1505.

\bibitem{M72}\label{M72} H. L. Montgomery, {\it The pair correlation of zeros of the zeta function}, Analytic Number Theory, Proc. Sympos. Pure Math., Vol. XXIV, St. Louis Univ., 1972, 181--193.

\bibitem{MO}\label{MO}
 H. L. Montgomery, A. M. Odlyzko, {\it Gaps between zeros of the zeta function},
    {Topics in classical number theory, {V}ol. {I}, {II}  ({B}udapest, 1981)},
{Colloq. Math. Soc. J\'anos Bolyai}, 34
1984, 1079--1106.

\bibitem{MV}
H. L. Montgomery, R. C. Vaughan, {\it Hilbert's inequality}, J. London Math. Soc. (2) {\bf 8} (1974), 73--82.

\bibitem{MW}
H. L. Montgomery, P. J. Weinberger, {\it Notes on small class numbers}, Acta Arith. {\bf 24} (1974),
529--542.

\bibitem{P}\label{P} S. Preobrazhenski\u\i,  {\it A small improvement in the gaps between consecutive zeros of the {R}iemann zeta-function},
Res. Number Theory \textbf{2} (2016), art. 28.

\bibitem{R14} \label{R14} M. Radziwi{\l}{\l}, {\it Gaps between zeros of $\zeta(s)$ and the distribution of zeros of $\zeta'(s)$}, Adv. Math. {\bf 257} (2014), 6--24.

\bibitem{RT}
B. Rodgers, T. Tao, {\it The de Bruijn--Newman constant is non-negative}, Forum Math. Pi 8 (2020), e6, 62 pp.


\bibitem{STT-B} A. Simonic, T. S. Trudgian, C. L. Turnage-Butterbaugh, {\it Some explicit and unconditional results on gaps between zeros of the Riemann zeta-function}, preprint, arxiv: 2010.10675.

\bibitem{So}\label{So}  K. Soundararajan, {\it On the distribution of gaps between zeros of the {R}iemann
              zeta-function},
              Q. J. Math.
{\bf 47} (1996), 383--387.

\bibitem{So09}\label{So09}  K. Soundararajan, {\it Moments of the {R}iemann zeta function}, Ann. of Math. {\bf 170} (2009), 981--993.

\bibitem{T}\label{T}
E. C. Titchmarsh, {\it The theory of the Riemann zeta-function}, $2$nd edition, edited and with a preface by D. R. Heath-Brown, The Clarendon Press, Oxford University Press (1986).


\bibitem{Wu}\label{Wu} X. Wu, {\it A note on the distribution of gaps between zeros of the
              {R}iemann zeta-function}, Proc. Amer. Math. Soc.
  {\bf 142} (2014), 851--857.


\end{thebibliography}
\end{document}